\documentclass[11pt,letterpaper,dvips]{article}

 \usepackage{amssymb,latexsym,amsmath,amsthm}

\newtheorem*{thm*}{Theorem}

\newtheorem{thm}{Theorem}

\newtheorem{lemma}{Lemma}

\begin{document}

\def\d{ \partial_{x_j} } 
\def\Na{{\mathbb{N}}}

\def\Z{{\mathbb{Z}}}

\def\IR{{\mathbb{R}}}

\newcommand{\E}[0]{ \varepsilon}

\newcommand{\la}[0]{ \lambda}

\newcommand{\s}[0]{ \mathcal{S}}

\newcommand{\AO}[1]{\| #1 \| }

\newcommand{\BO}[2]{ \left( #1 , #2 \right) }

\newcommand{\CO}[2]{ \left\langle #1 , #2 \right\rangle} 

\newcommand{\R}[0]{ \IR\cup \{\infty \} } 

\newcommand{\co}[1]{ #1^{\prime}} 

\newcommand{\p}[0]{ p^{\prime}} 

\newcommand{\m}[1]{   \mathcal{ #1 }} 

\newcommand{ \W}[0]{ \mathcal{W}}

\newcommand{ \A}[1]{ \left\| #1 \right\|_H }

\newcommand{\B}[2]{ \left( #1 , #2 \right)_H }

\newcommand{\C}[2]{ \left\langle #1 , #2 \right\rangle_{  H^* , H } }

 \newcommand{\HON}[1]{ \| #1 \|_{ H^1} }

\newcommand{ \Om }{ \Omega}

\newcommand{ \pOm}{\partial \Omega}

\newcommand{\D}{ \mathcal{D} \left( \Omega \right)}

\newcommand{\DP}{ \mathcal{D}^{\prime} \left( \Omega \right)  }

\newcommand{\DPP}[2]{   \left\langle #1 , #2 \right\rangle_{  \mathcal{D}^{\prime}, \mathcal{D} }}

\newcommand{\PHH}[2]{    \left\langle #1 , #2 \right\rangle_{    \left(H^1 \right)^*  ,  H^1   }    }

\newcommand{\PHO}[2]{  \left\langle #1 , #2 \right\rangle_{  H^{-1}  , H_0^1  }} 

 \newcommand{\HO}{ H^1 \left( \Omega \right)}

\newcommand{\HOO}{ H_0^1 \left( \Omega \right) }

\newcommand{\CC}{C_c^\infty\left(\Omega \right) }

\newcommand{\N}[1]{ \left\| #1\right\|_{ H_0^1  }  }

\newcommand{\IN}[2]{ \left(#1,#2\right)_{  H_0^1} }

\newcommand{\INI}[2]{ \left( #1 ,#2 \right)_ { H^1}} 

\newcommand{\HH}{   H^1 \left( \Omega \right)^* } 

\newcommand{\HL}{ H^{-1} \left( \Omega \right) }

\newcommand{\HS}[1]{ \| #1 \|_{H^*}}

\newcommand{\HSI}[2]{ \left( #1 , #2 \right)_{ H^*}}

\newcommand{\WO}{ W_0^{1,p}} 
\newcommand{\w}[1]{ \| #1 \|_{W_0^{1,p}}}  

\newcommand{\ww}{(W_0^{1,p})^*}   

\newcommand{\Ov}{ \overline{\Omega}}

\title{Regularity of stable solutions of a Lane-Emden type system}
\author{Craig Cowan\\
{\it\small Department of Mathematical Sciences}\\
{\it\small University of Alabama in Huntsville  }\\
{\it\small 258A Shelby Center}\\
\it\small Huntsville, AL 35899 \\ 
{\it\small ctcowan@stanford.edu} }

\maketitle


\vspace{3mm}

\begin{abstract}  We examine the system given by 

 \begin{eqnarray*} 
 \left\{ \begin{array}{lcl}
\hfill   -\Delta u    &=& \lambda (v+1)^p \qquad \Omega  \\
\hfill -\Delta v &=& \gamma (u+1)^\theta \qquad \Omega,  \\
\hfill u &=& v =0 \qquad \quad \pOm,
\end{array}\right. 
  \end{eqnarray*}    where $ \lambda,\gamma$ are positive parameters and where $ 1 < p \le \theta$ and where $ \Omega$ is a smooth bounded domain in $ \IR^N$.  We show the extremal solutions associated with the above system are bounded provided  
  \[ \frac{N}{2} < 1 + \frac{2(\theta+1)}{p\theta -1} \left(  \sqrt{  \frac{p \theta (p+1)}{\theta +1}   } + \sqrt{  \frac{p \theta (p+1)}{\theta +1}    - \sqrt{ \frac{p \theta (p+1)}{\theta +1}     }} \right).\]

\end{abstract}

\noindent
{\it \footnotesize 2010 Mathematics Subject Classification}. {\scriptsize }

\section{Introduction}

In this work we examine  the following  system: 
 \begin{eqnarray*} 
(N)_{\lambda,\gamma}\qquad  \left\{ \begin{array}{lcl}
\hfill   -\Delta u    &=& \lambda (v+1)^p \qquad \Omega  \\
\hfill -\Delta v &=& \gamma (u+1)^\theta \qquad \Omega,  \\
\hfill u &=& v =0 \qquad \quad \pOm,
\end{array}\right. 
  \end{eqnarray*}    
  where $\Omega$ is a bounded domain in $\IR^N$, $ \lambda, \gamma >0$ are positive parameters and where $ p,\theta >1$.     Our interest is in the regularity of the extremal solutions associated with $ (N)_{\lambda,\gamma}$.   In particular we are interested when the extremal solutions  of $(N)_{\lambda,\gamma}$  are bounded,  since one can then apply elliptic regularity theory to show the extremal solutions are classical solutions.  
 The nonlinearities we examine naturally fit into the following  class: \\

(R): \qquad $f$ is smooth, increasing, convex on  $\IR$ with $ f(0)=1$ and $ f$ is superlinear at $ \infty$ (i.e. $ \displaystyle \lim_{u \rightarrow \infty} \frac{f(u)}{u}=\infty$).

\subsection{Second order scalar case} 
For a nonlinearity $ f$ of type (R) consider the following  second order scalar analog of the above system  given by
\begin{equation*}
(Q)_\lambda \qquad  \left\{ 
\begin{array}{ll}
-\Delta u =\lambda f(u) &\hbox{in }\Omega \\
u =0 &\hbox{on } \pOm.
\end{array}
\right.
\end{equation*} 
This scalar equation is  now quite well understood whenever $ \Omega$ is a bounded smooth domain in $ \IR^N$. See, for instance, \cite{bcmr,BV,Cabre,CC,Martel,MP,Nedev}. We now list the  properties one comes to expect when studying $(Q)_\lambda$.  

\begin{itemize} \item  There exists a finite positive critical parameter $ \lambda^*$, called the \textbf{extremal parameter}, such that for all $ 0< \lambda < \lambda^*$ there exists a a smooth \textbf{minimal solution} $ u_\lambda$ of $ (Q)_\lambda$.   By minimal solution, we mean here that if $ v$ is another solution of $ (Q)_\lambda$ then $v \ge u_\lambda$ a.e. in $ \Omega$.  

\item For each $ 0< \lambda < \lambda^*$ the minimal solution $ u_\lambda$ is \textbf{semi-stable} in the sense that 
\[ \int_\Omega \lambda f'(u_\lambda) \psi^2 dx \le \int_\Omega | \nabla \psi|^2 dx, \qquad \forall \psi \in H_0^1(\Omega),\]  
and is unique among all the weak semi-stable solutions. 
 
\item The map $ \lambda \mapsto u_\lambda(x)$ is increasing on $ (0,\lambda^*)$ for each $ x \in \Omega$.    This allows one to define $ u^*(x):= \lim_{\lambda \nearrow \lambda^*} u_\lambda(x)$, the so-called {\bf extremal solution}, which can be shown to be a weak solution of $ (Q)_{\lambda^*}$.    In addition one can show that $ u^*$ is the unique weak solution of $(Q)_{\lambda^*}$. See \cite{Martel}. 
\item There are no solutions of $ (Q)_{\lambda}$ (even in a very weak sense) for $ \lambda > \lambda^*$.   

\end{itemize} 

A question which has attracted a lot of attention is whether the extremal function $ u^*$ is a classical solution of $(Q)_{\lambda^*}$. This is of interest since one can then apply the results from \cite{CR}  to start a second branch of solutions emanating from $(\lambda^*, u^*)$.   
The answer typically depends on the nonlinearity $f$, the dimension $N$ and the geometry of the domain $ \Omega$.  We now list some known results.
 
 \begin{itemize} \item  \cite{CR} Suppose $ f(u)=e^u$.  If  $ N <10$ then $ u^*$ is bounded.  For $ N \ge 10$ and $ \Omega$ the unit ball $ u^*(x)=-2 \log(|x|)$.

 \item \cite{CC} Suppose $ f$ satisfies (R) but without the convexity assumption and $ \Omega$ is the unit ball.   Then $ u^*$ is bounded for $ N <10$.  In view of the above result this is optimal.

 \item On general domains, and if $f$ satisfies (R), then $ u^*$ is bounded for $ N \le 3$ \cite{Nedev}.   Recently this has been improved to $ N \le 4$ provided the domain is convex (again one can drop the convexity assumption on $f$), see \cite{Cabre}.

  \end{itemize}

We now examine the generalization of $(N)_{\lambda,\gamma}$ given by

 \begin{eqnarray*} 
(P)_{\lambda,\gamma}\qquad  \left\{ \begin{array}{lcl}
\hfill   -\Delta u    &=& \lambda f(v)\qquad \Omega  \\
\hfill -\Delta v &=& \gamma g(u)   \qquad \Omega,  \\
\hfill u &=& v =0 \qquad \pOm,
\end{array}\right.
  \end{eqnarray*}    where $ f$ and $ g$ satisfy (R).      
 Define  $ \mathcal{Q}=\{ (\lambda,\gamma): \lambda, \gamma >0 \}$, 
\[ \mathcal{U}:= \left\{ (\lambda,\gamma) \in \mathcal{Q}: \mbox{ there exists a smooth solution $(u,v)$ of $(P)_{\lambda,\gamma}$} \right\},\]   and  set
 $ \Upsilon:= \partial \mathcal{U} \cap \mathcal{Q}$. Note that $\Upsilon$ is the analog of $ \lambda^*$ for the above system.  
  A generalization of $(P)_{\lambda,\gamma}$ was examined in \cite{Mont} and many results were obtained, including

\begin{thm*} (Montenegro, \cite{Mont}) Suppose $f$ and $g$ satisfy (R).  Then 
\begin{enumerate} \item $\mathcal{U}$ is nonempty.   

\item For all $ (\lambda,\gamma) \in \mathcal{U}$ there exists a smooth, minimal solution of $(P)_{\lambda,\gamma}$. 

\item For each $ 0 < \sigma < \infty$ there is some $ 0 < \lambda^*_\sigma < \infty$ such that   $ \mathcal{U} \cap \{ (\lambda,\sigma \lambda): 0 < \lambda \}$  is given by  $ \{ (\lambda, \sigma \lambda): 0 < \lambda < \lambda_\sigma^* \} \cup \mathcal{H}$
 where $\mathcal{H}$ is either the empty set or $ \{ (\lambda_\sigma^*, \sigma \lambda_\sigma^*) \}$.  The map $\sigma \mapsto \lambda_\sigma^*$ is bounded on compact subsets of $(0,\infty)$.   Fix $ 0 < \sigma < \infty$ and let $ (u_\lambda,v_\lambda)$ denote the smooth minimal solution of $(P)_{\lambda,\sigma \lambda}$ for $ 0 <\lambda < \lambda_\sigma^*$.  Then $ u_\lambda(x),v_\lambda(x)$ are increasing in $ \lambda$ and hence 
 \[ u^*(x):= \lim_{\lambda \nearrow \lambda_\sigma^*} u_\lambda(x), \quad v^*(x):= \lim_{\lambda \nearrow \lambda_\sigma^*} u_\lambda(x),\] are well defined and can be shown to be a weak solution of $(P)_{\lambda_\sigma^*, \sigma \lambda_\sigma^*}$.  

\end{enumerate}  
\end{thm*}    

  Our notation will vary slightly from above.  Let $ (\lambda^*,\gamma^*) \in \Upsilon$ and set $ \sigma:= \frac{\gamma^*}{\lambda^*}$.  Define $ \Gamma_\sigma:=\{ (\lambda,\sigma \lambda):  \frac{\lambda^*}{2} < \lambda < \lambda^* \}$ and we let $ (u^*,v^*)$, called the \textbf{extremal solution} associated with $(P)_{\lambda^*,\gamma^*}$, be the pointwise limit of the minimal solutions along the ray $ \Gamma_\sigma$ as $ \lambda \nearrow \lambda^*$. As mentioned above $(u^*,v^*)$ is a weak solution of $(N)_{\lambda^*,\gamma^*}$ in a suitable sense.

 The following result shows that the minimal solutions are stable in some suitable sense and this will be crucial in obtaining regularity of the extremal solutions associated with $ (N)_{\lambda,\gamma}$.     
\begin{thm*} (Montenegro \cite{Mont})   Let $ (\lambda,\gamma) \in \mathcal{U} $ and let $ (u,v)$ denote the minimal solution of $(P)_{\lambda,\gamma}$.  Then $(u,v)$ is semi-stable in the sense that there is some smooth $ 0 < \zeta,\chi \in H_0^1(\Omega)$ and $  0 \le \eta $  such that 
\begin{equation} \label{sta}
 -\Delta \zeta = \lambda  f'(v) \chi + \eta \zeta, \qquad -\Delta \chi = \gamma g'(u) \zeta+ \eta \chi, \qquad  \mbox{ in } \Omega.
 \end{equation} 
\end{thm*}    

We give an alternate proof of a result which is slightly different than the above one,  but which is sufficient for our purposes.   Fix $ (\lambda^*, \gamma^*) \in \Upsilon$, $ \sigma:= \frac{\gamma^*}{\lambda^*}$ and let $ (u_\lambda,v_\lambda)$ denote minimal solution of $(P)_{\lambda, \gamma}$ on the ray $ \Gamma_\sigma$.     Taking a derivative in $ \lambda$ of $(P)_{\lambda,\sigma \lambda}$ shows that 
\[ -\Delta \tilde{\zeta} =  \lambda f'(v_\lambda) \tilde{\chi} + f(v_\lambda), \quad -\Delta \tilde{\chi} =   \lambda \sigma g'(u_\lambda) \tilde{\zeta} +\sigma g(u_\lambda) \quad \mbox{ in $ \Omega$,}   \]  where $ \tilde{\zeta}:=\partial_\lambda u_\lambda$ and $ \tilde{\chi}:= \partial_\lambda v_\lambda$.  
 Using the monotonicity of $ u_\lambda, v_\lambda$ and the maximum principle shows that $ \tilde{\zeta}, \tilde{\chi} >0$.

We now recall some known results regarding the regularity of the extremal solutions associated with various systems.  In what follows $ \Omega$ is a bounded domain in $ \IR^N$. 
  
\begin{itemize}  

\item  In \cite{craig0} the following system 
   \begin{eqnarray*} 
(E)_{\lambda,\gamma}\qquad  \left\{ \begin{array}{lcl}
\hfill   -\Delta u    &=& \lambda e^v \qquad \Omega  \\
\hfill -\Delta v &=& \gamma e^u  \qquad \Omega,  \\
\hfill u &=& v =0 \quad  \pOm,
\end{array}\right.
  \end{eqnarray*}  was examined.     It was shown that if $ 3 \le N \le 9$ and 
  \[ \frac{N-2}{8} < \frac{\gamma^*}{\lambda^*} < \frac{8}{N-2},\]  then the extremal solution $(u^*,v^*)$ is bounded.     Note that  not only does the dimension $N$ play a role but  how close $ (\lambda^*,\gamma^*)$ are to the diagonal $ \gamma=\lambda$ plays a role. When $ \gamma=\lambda$ one can show that the above system reduces to the scalar equation $ -\Delta u = \lambda e^u$.   We remark that we were unable to extend the methods used in \cite{craig0} to handle $(N)_{\lambda,\gamma}$ except in the case where $ p=\theta$.  
  
  \item  In \cite{craig2} the system 
\[ (P')_{\lambda,\gamma} \qquad -\Delta u = \lambda F(u,v), \qquad -\Delta v = \gamma G(u,v) \qquad \mbox{in $ \Omega$,} \]  with $ u=v=0$ on $ \pOm$ was examined  examined in the cases where $ F(u,v)=f'(u) g(v), \;  G(u,v)= f(u)g'(v)$  (resp. $ F(u,v)=f(u) g'(v), \; G(u,v) = f'(u) g(v)$) and were denoted by $ (G)_{\lambda, \gamma}$  (resp. $(H)_{\lambda,\gamma}$).  It was shown that the extremal solutions associated with $(G)_{\lambda,\gamma}$ were bounded provided $ \Omega$ was a convex domain in $ R^N$ where $ N \le 3$ and $f$ and $g$ satisfied conditions similar to (R).  Regularity results regarding $(H)_{\lambda,\gamma}$ we also obtained in the case where at least one of of $ f$ and $g$ were explicit nonlinearities given by $ (u+1)^p$ or $ e^u$.

 \end{itemize}

\section{Main Results}

  We now state our main results.

\begin{thm} \label{main1} Suppose that $ 1 < p \le \theta$,    $(\lambda^*,\gamma^*) \in \Upsilon$ and let  $ (u^*,v^*)$ denote the extremal solution associated with $ (N)_{\lambda^*,\gamma^*}$.  Suppose that 
\[ \frac{N}{2} < 1 + \frac{2(\theta+1)}{p\theta -1} \left(  \sqrt{  \frac{p \theta (p+1)}{\theta +1}   } + \sqrt{  \frac{p \theta (p+1)}{\theta +1}    - \sqrt{ \frac{p \theta (p+1)}{\theta +1}     }} \right).\]   Then $ u^*,v^*$ are bounded. 

\end{thm}

  There are two main steps in proving the above theorems.     We first show that minimal solutions of $(P)_{\lambda,\gamma}$, which are semi-stable in the sense of (\ref{sta}),  satisfy a stability inequality which is reminiscent of semi-stability in the sense of the second order scalar equations.    This is given by Lemma \ref{stabb}.   
  
The second  ingredient will be a  pointwise comparison result between $u$ and $v$, given in Lemma \ref{pointwise}. 
 We remark that this was motivated by  \cite{souplet_4} and a similar result was used  in \cite{craig1}.     

\begin{lemma} \label{stabb}  

 Let $(u,v)$ denote a semi-stable solution of $(P)_{\lambda,\gamma}$ in the sense of (\ref{sta}).  Then 
\begin{equation} \label{first}
2 \sqrt{\lambda \gamma} \int \sqrt{f'(v) g'(u)} \phi \psi \le \int | \nabla \phi|^2 + | \nabla \psi|^2,
\end{equation} for all $ \phi,\psi \in H_0^1(\Omega)$.  Taking $ \phi=\psi$ gives 
\begin{equation} \label{second}
 \sqrt{\lambda \gamma} \int \sqrt{f'(v) g'(u)} \phi^2 \le \int | \nabla \phi|^2 
\end{equation} for all $ \phi \in H_0^1(\Omega)$. 

\end{lemma}

\begin{proof} Since $(u,v)$ is a semi-stable solution of $(P)_{\lambda,\gamma}$ there is some $ 0 < \zeta, \chi \in H_0^1(\Omega)$ smooth such that 
\[ \frac{-\Delta \zeta}{\zeta} \ge \lambda f'(v) \frac{\chi}{\zeta}, \qquad \frac{-\Delta \chi}{\chi} \ge \gamma g'(u) \frac{\zeta}{\chi}, \qquad \mbox{ in $ \Omega$.}\]  Let $ \phi,\psi \in C_c^\infty(\Omega)$ and multiply the first equation by $ \phi^2$ and the second by $\psi^2$  and integrate over $ \Omega$ to arrive at 
\[ \int \lambda f'(v) \frac{\chi}{\zeta} \phi^2 \le \int | \nabla \phi|^2, \qquad \int \gamma g'(u) \frac{\zeta}{\chi} \psi^2 \le \int | \nabla \psi|^2,\]  where we have utilized the result that  for any sufficiently smooth $ E>0$ we have 
\[ \int \frac{-\Delta E}{E} \phi^2 \le \int | \nabla \phi|^2,\] for all $ \phi \in C_c^\infty(\Omega)$.   We now add the inequalities to obtain 
\begin{equation} \label{thing} 
 \int (\lambda f'(v)  \phi^2)   \frac{\chi}{\zeta}  + (\gamma g'(u)  \psi^2 )\frac{\zeta}{\chi} \le \int | \nabla \phi|^2 + | \nabla \psi|^2.
\end{equation}   Now note that 
\[ 2 \sqrt{ \lambda \gamma f'(v) g'(u)} \phi \psi \le 2t \lambda f'(v) \phi^2 + \frac{1}{2t} \gamma g'(u) \psi^2,\] for any $ t>0$.  Taking $ 2t = \frac{\chi(x)}{\zeta(x)}$ gives 
\[ 2 \sqrt{ \lambda \gamma f'(v) g(u)} \phi \psi \le (\lambda f'(v)  \phi^2)   \frac{\chi}{\zeta}  + (\gamma g'(u)  \psi^2 )\frac{\zeta}{\chi},\]  and putting this back into (\ref{thing}) gives the desired result.

\end{proof}

\begin{lemma} \label{pointwise} Let $ (u,v)$ denote a smooth solution of $(N)_{\lambda,\gamma}$ and suppose  that $ \theta \ge p >1$.   Define 
\[ \alpha:= \max \left\{ 0, \left(  \frac{\gamma (p+1)}{\lambda (\theta+1) } \right)^\frac{1}{p+1}-1 \right\} .\] Then 
\begin{equation} \label{point}  
\lambda (\theta +1) (v+1 +\alpha)^{p+1} \ge \gamma (p+1) (u+1)^{ \theta +1} \quad \mbox{ in $ \Omega$}.
\end{equation}   

\end{lemma}

\begin{proof}  Let $(u,v)$ denote a smooth solution of $(N)_{\lambda,\gamma}$ and define $ w:= v+1+\alpha - C(u+1)^t$ where 
\[ C:= \left( \frac{\gamma (p+1)}{\lambda ( \theta +1)} \right)^\frac{1}{p+1} \mbox{ and }  \quad t:= \frac{\theta +1}{p+1} \ge 1.\]   

 Note that $ w \ge 0$ on $ \pOm$ and define $ \Omega_0:= \{ x \in \Omega: w(x) < 0 \}$.  If $ \Omega_0$ is empty then we are done so we suppose that $ \Omega_0$ is nonempty.  Note that since $ w \ge 0$ on $ \pOm$ we have $ w = 0$ on $ \pOm_0$.   A computation shows that 
\[- \Delta w = \gamma (u+1)^\theta - C t (u+1)^{t-1} \lambda (v+1)^p + C t (t-1) (u+1)^{t-2} | \nabla u|^2, \qquad  \mbox{ in $ \Omega$,} \]  and since $ t \ge 1$ we have

\[ -\Delta w  \ge   \gamma (u+1)^\theta - C t (u+1)^{t-1} \lambda (v+1)^p  \quad \mbox{ in $ \Omega$.} \] 

  Note that we have, by definition,   
\[ v+1 \le v+1+\alpha < C(u+1)^t \qquad \mbox{ in $ \Omega_0$,} \]  and so we have 
\[ -\Delta w  \ge \gamma (u+1)^\theta - C^{p+1} t \lambda (u+1)^{tp+t-1} \quad \mbox{ in $ \Omega_0$},\]   but the right hand side of this is zero and hence we have $ -\Delta w \ge 0$ in $ \Omega_0$ with $ w = 0 $ on $ \pOm_0$ and hence $ w \ge  0$ in $ \Omega_0$, which is a contradiction.  So $ \Omega_0$ is empty. 
  
\end{proof}

\textbf{Proof of Theorem \ref{main1}.}    Let $ (\lambda^*,\gamma^*) \in \Upsilon$ and let $ \sigma:=\frac{\gamma^*}{\lambda^*}$ and suppose that $ (u,v)$ denotes a minimal solution of $(N)_{\lambda,\gamma}$ on the ray $ \Gamma_\sigma$.   Put $ \phi:= (v+1)^t-1$, where  $\frac{1}{2} < t$, into (\ref{second})   to obtain 
\[ \sqrt{\lambda \gamma p \theta} \int (v+1)^\frac{p-1}{2} (u+1)^\frac{\theta-1}{2} ( (v+1)^t-1)^2 \le t^2 \int (v+1)^{2t-2} | \nabla v|^2,\] and multiply $(N)_{\lambda,\gamma}$ by $ (v+1)^{2t-1}-1$ and integrate by parts to obtain 
\[ t^2 \int (v+1)^{2t-2} | \nabla v|^2 = \frac{t^2 \gamma}{2t-1} \int (u+1)^\theta ( (v+1)^{2t-1}-1).\]   Equating these and expanding the squares and dropping some positive terms gives 
\begin{eqnarray} \label{poo-poo} 
&& \sqrt{ \lambda \gamma p \theta} \int (v+1)^\frac{p-1}{2} (u+1)^\frac{\theta-1}{2} (v+1)^{2t}  \nonumber \\
 &   &  \qquad  \le    \frac{t^2 \gamma}{2t-1} \int (u+1)^\theta (v+1)^{2t-1} \nonumber \\
 && \qquad \quad   + 2 \sqrt{ \lambda \gamma p \theta} \int (v+1)^\frac{p-1}{2} (u+1)^\frac{\theta-1}{2} (v+1)^t.  
\end{eqnarray}
 We now use Lemma \ref{pointwise} to get a lower bound for  
\[ I:= \int (v+1)^\frac{p-1}{2} (u+1)^\frac{\theta-1}{2} (v+1)^{2t},\]  but we need to rework the pointwise estimate (\ref{point}) first.  From (\ref{point}) we have 
\[ \sqrt{  \frac{\gamma (p+1)}{\lambda (\theta+1)}} (u+1)^\frac{\theta+1}{2} \le (v+1 + \alpha )^\frac{p+1}{2},\]  and for all $ \delta >0$ there is some $ C(\delta) >0$ such that 
\[ (v+1 + \alpha)^\frac{p+1}{2} \le (1+\delta) (v+1)^\frac{p+1}{2} + C(\delta ) \alpha^\frac{p+1}{2}.\]   From this we see that there is some $ C_1= C_1(\delta,p,\alpha) $ such that 
\[ (v+1)^\frac{p+1}{2} \ge  \sqrt{  \frac{\gamma (p+1)}{\lambda (\theta+1)}} \frac{ (u+1)^\frac{\theta+1}{2} }{1+\delta} - C_1.\]   We now rewrite $I$ as 
\[ I= \int (u+1)^\frac{\theta-1}{2} (v+1)^{2t-1} (v+1)^\frac{p+1}{2},\] and use the above estimate to show that 
\begin{equation} \label{I}
I \ge \sqrt{ \frac{ \gamma (p+1)}{\lambda (\theta+1)}} \frac{1}{\delta +1} \int (u+1)^\theta (v+1)^{2t-1} - C_1 \int (u+1)^\frac{\theta-1}{2} (v+1)^{2t-1}.
\end{equation}  We now return to (\ref{poo-poo})  and write the left hand side, where $ \E>0$ is small,   as 
\[ \E \sqrt{ \lambda \gamma p \theta} I + (1-\E) \sqrt{ \lambda \gamma p \theta} I,\] and we leave the first term alone and we use the above lower estimate for $I$ on the second term.   Putting this back into (\ref{poo-poo}) and after some rearranging one arrives at 

\begin{eqnarray} \label{pppp} 
\E \sqrt{ \lambda \gamma p \theta} I  &+ &  \gamma K \int (u+1)^\theta (v+1)^{2t-1} \nonumber \\ 
& \le & 2 \sqrt{ \lambda \gamma p \theta} \; I_1 + (1-\E) \sqrt{ \lambda \gamma p \theta} C_1  \; I_2
\end{eqnarray} where 
\[ K:= \frac{(1-\E)}{1+\delta} \sqrt{  \frac{p \theta (p+1)}{\theta +1}} - \frac{t^2}{2t-1},\]
\[ I_1:= \int (v+1)^\frac{p-1}{2} (u+1)^\frac{\theta-1}{2} (v+1)^t,  \qquad \mbox{ and } \] 
\[ I_2:= \int (u+1)^\frac{\theta-1}{2} (v+1)^{2t-1}.\]  
For the moment we assume the following \textbf{claims:}   for all $ T>1$ and $ k >1$ 
\begin{eqnarray} \label{claim1} 
I_2 & \le &  \frac{1}{T^\frac{\theta+1}{2}} \int (u+1)^\theta (v+1)^{2t-1}  + | \Omega | T^\frac{\theta-1}{2} k^{2t-1} \nonumber \\
&& + \frac{1}{k^\frac{p+1}{2}} \int (u+1)^\frac{\theta-1}{2} (v+1)^{ \frac{p-1}{2} +2t},
\end{eqnarray}  and 
\begin{eqnarray} \label{claim2} 
I_1 & \le & \frac{1}{T^t} \int (u+1)^\frac{\theta-1}{2} (v+1)^{ \frac{p-1}{2} +2t} + | \Omega| T^{ \frac{p-1}{2} +t} k^\frac{\theta-1}{2} \nonumber \\
&& + \frac{ T^{ |   \frac{p+1}{2}-t |}}{k^\frac{\theta+1}{2}} \int (u+1)^\theta (v+1)^{2t-1}.
\end{eqnarray}   Putting (\ref{claim1}) and (\ref{claim2}) back into (\ref{pppp}) one arrives at an estimate of the form 

\begin{equation} \label{final} 
K_1 I + K_2 \int (u+1)^\theta (v+1)^{2t-1} \le C(\E,p,\theta, T, k,\delta)
\end{equation} where 
\[ K_1:= \E \sqrt{ \lambda \gamma p \theta} - \frac{ 2 \sqrt{ \lambda \gamma p \theta}}{T^t} - \frac{(1-\E) \sqrt{\lambda \gamma p \theta}\;  C_1}{k^\frac{p+1}{2}}, \quad \mbox{ and } \] 
 
\[ K_2:= \gamma K - \frac{   2 \sqrt{ \lambda \gamma p \theta} \; T^{ | \frac{p+1}{2} -t |}}{k^\frac{\theta+1}{2}} - \frac{(1-\E) \sqrt{ \lambda \gamma p \theta} \; C_1}{T^\frac{\theta+1}{2}},    \]  and where $ C(\E,p,\theta, T, k,\delta)$ is a positive finite constant which is uniform on the ray $ \Gamma_\sigma$.    Define 
\[ t_0:= \sqrt{ \frac{ p \theta (p+1)}{\theta+1}} + \sqrt{   \frac{ p \theta (p+1)}{\theta+1} -  \sqrt{ \frac{ p \theta (p+1)}{\theta+1}}}   \] and note that $ t_0>1$ for all $ p,\theta >1$.  Fix $ 1 <t< t_0$ and hence 
\[  \sqrt{ \frac{ p \theta (p+1)}{\theta+1}} - \frac{t^2}{2t-1} >0.\] We now fix $ \E>0$ and $ \delta >0$ sufficiently small such that $ K >0$.  
 We now fix $ T>1$ sufficiently large such that  
\[ \E \sqrt{ \lambda \gamma p \theta} - \frac{ 2 \sqrt{ \lambda \gamma p \theta}}{T^t}  \] (this is the first two terms from $K_1$)  and 
\[ \gamma K - \frac{(1-\E) \sqrt{ \lambda \gamma p \theta} \; C_1}{T^\frac{\theta+1}{2}} \]
(the first and third terms from $K_2$) are positive and  bounded away from zero on the the ray $ \Gamma_\sigma$.   
We now take $ k>1$ sufficiently big such that $ K_1,K_2 $ are positive and  bounded away from zero on the ray $ \Gamma_\sigma$ and hence we have estimates of the form: for all $ 1  < t<t_0$ there is some $ C_t >0$ such that 
\begin{equation} \label{ll}
 \int (u+1)^\theta (v+1)^{2t-1} \le C_t,
\end{equation} 
where $ C_t $ is some finite uniform constant  on the ray $ \Gamma_\sigma$.  Using the pointwise lower estimate (\ref{point}) for $ v+1$ gives: for all $ 1 <t<t_0$ there is some  $ \tilde{C}_t <\infty$, uniform along the ray $ \Gamma_\sigma$,  such that 

\begin{equation} \label{zzzz}
\int (u+1)^{ \theta + ( \frac{ (\theta+1) (2t-1)}{p+1}} \le \tilde{C}_t, 
\end{equation}  and hence this estimate also holds if one replaces $u$ with $ u^*$.    We now let $ 1 <t<t_0$ and note that 
\[ \frac{1}{\gamma} \int | \nabla v|^2 = \int (u+1)^\theta v  \le  \int (u+1)^\theta (v+1) \le\int (u+1)^\theta (v+1)^{2t-1} \le C_t,\] by (\ref{ll}) and hence we can pass to the limit and see that $ v^* \in H_0^1(\Omega)$.   We now proceed to show that $ v^*$ is bounded in low dimensions.    First note that 
\[ \frac{-\Delta v^*}{\gamma^*} = (u^*+1)^\theta = \frac{ (u^*+1)^\theta}{v^*+1} v^* + \frac{ (u^*+1)^\theta}{v^*+1} \quad \mbox{in $ \Omega$.}\]    To show that $ v^*$ is bounded it is sufficient, since $ v^* \in H_0^1(\Omega)$, to show that  $ \frac{(u^*+1)^\theta}{v^*+1} \in L^T(\Omega)$ for some $T> \frac{N}{2}$.   Using (\ref{point}) and passing to the limit one sees there is some $ C>0$ such that 
\[ \frac{ (u^*+1)^\theta}{v^*+1} \le C (u^*+1)^\frac{p \theta-1}{p+1} +C \qquad \mbox{ in $ \Omega$,} \]  and so $ \frac{ (u^*+1)^\theta}{v^*+1} \in L^T(\Omega)$ for some $T>\frac{N}{2}$ provided 
\[ \frac{(p \theta -1)}{p+1} \frac{N}{2} < \theta + \frac{( \theta +1) (2 t_0-1)}{p+1},\] after considering (\ref{zzzz}).  This rearranges into 
\[ \frac{N}{2} < 1 + 2 \frac{(\theta+1)}{p+1} t_0,\] which is the desired result.  We now use $ (N)_{\lambda,\gamma}$ and elliptic regularity to see that $ u^*$ is also bounded.

\hfill $ \Box$

\textbf{Proof of Claims (\ref{claim1}) and (\ref{claim2}).}  We first prove (\ref{claim1}). We write $I_2$ as 
\[ I_2= \int_{u+1 \ge T} + \int_{ u+1 <T, v+1 \le k} + \int_{u+1<T, v+1 >k},\] where the integrands are the same as in $I_2$.  Note that the first integral is less than or equal 
\[ \int_{u+1 >T} (u+1)^\frac{\theta-1}{2} \left( \frac{u+1}{T} \right)^\frac{\theta+1}{2} (v+1)^{2t-1} \le \frac{1}{T^\frac{\theta+1}{2}} \int (u+1)^\theta (v+1)^{2t-1}.\]  The second integral is trivial to get upper estimate on.  One estimates the third integral in the same way as the first to see that 
\[ \int_{u+1 <T, v+1 >k} \le \frac{1}{k^\frac{p+1}{2}} \int (u+1)^\frac{\theta-1}{2} (v+1)^{ \frac{p-1}{2} +2t}.\]   Combining these estimates gives (\ref{claim1}).  We now prove (\ref{claim2}).  We write 
\[ I_1= \int_{v+1 \ge T} + \int_{v+1 <T, u+1 <k} + \int_{v+1<T,u+1 \ge k},\] where the integrands are the same as $I_1$.  Note that the first integral is less than or equal 
\[ \int_{v+1>T} (v+1)^{ \frac{p-1}{2} +t} \left( \frac{v+1}{T} \right)^t (u+1)^\frac{\theta-1}{2},\]  and this is less than or equal 
\[ \frac{1}{T^t} \int (v+1)^{ \frac{p-1}{2} + 2t} (u+1)^\frac{\theta-1}{2}.\]  The second integral is easily estimated.   We rewrite the third integral as  
 \[ \int_{v+1<T, u+1 \ge k} (v+1)^{2t-1} (u+1)^\theta \left( (v+1)^{ \frac{p+1}{2}-t} (u+1)^\frac{ -\theta-1}{2} \right),\]  and we now estimate the terms inside the bracket in the obvious manner.  Combining these gives (\ref{claim2}).

\hfill $ \Box$

\end{document}